\documentclass{amsart}
\usepackage{geometry}
\usepackage{amssymb,latexsym,enumitem,bbm}
\usepackage[noadjust]{cite}
\usepackage{setspace}
\usepackage{color}
\usepackage{hyperref}
\usepackage{bbm}

\newcommand{\R}{\mathbb{R}}
\newcommand{\N}{\mathbb{N}}
\newcommand{\F}{\mathcal{F}}

\newcommand{\Sc}{\mathcal{S}}
\newcommand{\Exp}{\mathbb{E}}
\newcommand{\oo}{\mathcal{O}}
\newcommand{\K}{\mathcal{K}}

\newcommand{\inpr}[3][]{\left\langle#2 \,,\, #3\right\rangle_{#1}}

\numberwithin{equation}{section}

\newtheorem{theorem}{Theorem}[section]

\newtheorem{lemma}[theorem]{Lemma}
\newtheorem{proposition}[theorem]{Proposition}
\newtheorem{remark}[theorem]{Remark}

\allowdisplaybreaks

\title[Parametric family of SDEs]{Parametric family of SDEs driven by L\'evy noise} 

\author{Suprio Bhar}
\address{Suprio Bhar, Tata Institute of Fundamental Research, Centre For Applicable Mathematics,Post Bag No 6503, GKVK Post Office, Sharada Nagar, Chikkabommsandra, Bangalore 560065, India.}
\email{suprio@tifrbng.res.in, suprio.bhar@gmail.com}

\author{Barun Sarkar}
\address{Barun Sarkar, Indian Statistical Institute Bangalore Centre, 8th Mile Mysore Road, Bangalore 560059, India.}
\email{barunsarkar.math@gmail.com}

\begin{document}

\begin{abstract} 
In this article we study the existence and uniqueness of strong solutions of a class of parameterized family of SDEs driven by L\'evy noise. These SDEs occurs in connection with a class of stochastic PDEs, which take values in the space of tempered distributions $\Sc^\prime$. This
correspondence for diffusion processes was proved in [Rajeev, \emph{Translation invariant diffusion in the space of 
tempered distributions}, Indian J. Pure Appl. Math. \textbf{44} (2013), no.~2,
  231--258]. 
\end{abstract}
\keywords{$\mathcal{S}^\prime$ valued process, L\'{e}vy processes, Hermite-Sobolev space, Strong solution}
\subjclass[2010]{60G51, 60H10}

\maketitle

\section{Introduction}\label{S:1}
Given a complete filtered probability space $\big(\Omega,\F,\{\F_t\}_{t\geq0},P\big)$ satisfying the usual conditions, we consider the existence and uniqueness of strong solutions of a class of stochastic differential equations (SDEs) in $\R^d$, viz.
\begin{equation}
\begin{split}
dU_{t} &=\bar b(U_{t-};\xi)dt+ \bar\sigma(U_{t-};\xi)\cdot dB_t +\int_{(0 < |x| < 1)}  \bar F(U_{t-},x;\xi)\, \widetilde
N(dtdx)\\
&+\int_{(|x| \geq  1)} \bar G(U_{t-},x;\xi) \,
N(dtdx), \quad t\geq 0 \\
U_0&=\kappa,
\end{split}
\end{equation}
where
\begin{enumerate}[label=(\roman*)]
\item $\{B_t\}$ denotes an $\R^d$ valued standard Brownian motion and $N$ a Poisson random measure driven by a L\'evy measure $\nu$. $\widetilde N$ denotes the corresponding compensated random measure. We also assume that $B$ and $N$ are independent.
\item The parameter $\xi$ is an $\F_0$-measurable random variable and takes values in some specific Hilbert space, viz. the Hermite-Sobolev spaces (see Section \ref{S:2}). The random variable $\kappa$ is $\R^d$ valued and $\F_0$-measurable. Unless stated otherwise, $\xi$ and $\kappa$ will be taken to be independent of the noise $B$ and $N$.
\item The coefficients $\bar\sigma, \bar b, \bar F$ and $\bar G$ are defined in terms of $\sigma, b, F$ and $G$ which are the coefficients of an associated stochastic PDE, see for example \cite[p. 524]{MR3647067}, \cite[p. 170]{MR3687773}, \cite[p. 237]{MR3063763}. Note that the coefficients are allowed to be $\F_0$ measurable.
\end{enumerate}

Such SDEs occurs in connection with a class of stochastic PDEs whose solutions take values in the space of tempered distributions $\Sc^\prime$, see for example \cite{MR3063763, MR3647067, JOTP-erratum, MR3687773}. We can study the ergodicity/stationarity properties of these stochastic PDEs via the corresponding finite dimensional SDEs. A standard approach in proving the existence and uniqueness results for SDEs is to assume that the coefficients are Lipschitz (see \cite{MR1011252, MR2020294, MR2512800, MR2560625, MR1398879, MR2001996, MR1121940} and the references therein). The goal of this article is to describe hypotheses, which include appropriate parameterized versions of Lipschitz regularity of the coefficients and prove in detail the existence and uniqueness results.

We now describe the layout of the paper. In Section \ref{S:2}, we describe the space of Schwartz class functions $\Sc$ and its dual, the space of tempered distributions $\Sc^\prime$. We also recall definitions of the Hermite-Sobolev spaces $\Sc_p, p \in \R$.

In Section \ref{S:3}, we state the notation and hypotheses followed in the rest of the article. In Theorem \ref{nrm-bd-rndm-inl} the existence and uniqueness result is proved for the reduced equation with `global Lipschitz' coefficients and then in Theorem \ref{interlacing-global-sde} proved for the general case (i.e. involving the large jumps)  by an interlacing technique. In Theorem \ref{nrm-sqre-rndm-inl-fnl}, we prove the result for `local Lipschitz' coefficients.

In \cite[Proposition 3.7]{Levy-SPDE}, it is proved that the `local Lipschitz' regularity of the coefficients $\bar\sigma, \bar b, \bar F$ follow from explicit regularity assumptions on $\sigma, b, F$ provided other hypotheses are satisfied. Furthermore, the existence and uniqueness problems for the corresponding SPDEs are studied in \cite{Levy-SPDE}.

\section{Topology on Schwartz space}\label{S:2}
Let $\Sc$ be the space of rapidly decreasing smooth functions on $\R^d$ with dual $\Sc^\prime$, the space of tempered distributions (see \cite{MR771478}). Let $\mathbb{Z}^d_+:=\{n=(n_1,\cdots, n_d): \; n_i \text{ non-negative integers}\}$. If $n\in\mathbb{Z}^d_+$, we define $|n|:=n_1+\cdots+n_d$.

For $p \in \R$, consider the increasing norms $\|\cdot\|_p$, defined by the inner
products
\begin{equation}
\langle f,g\rangle_p:=\sum_{n\in\mathbb{Z}^d_+}(2|n|+d)^{2p}\langle f,h_n\rangle\langle g,h_n\rangle,\ \ \ f,g\in\Sc.
\end{equation}
In the above equation, $\{h_n: n\in\mathbb{Z}^d_+\}$ is an orthonormal basis for $\mathcal{L}^2(\R^d,dx)$ given by the Hermite functions and $\langle\cdot,\cdot\rangle$ is the usual
inner product in $\mathcal{L}^2(\R^d,dx)$. For $d=1$,
$h_n(t) :=(2^n n!\sqrt{\pi})^{-1/2}\exp\{-t^2/2\}H_n(t)$, where $H_n, t \in \R$ are the Hermite polynomials (see \cite{MR771478}). For $d > 1$, $h_n(x_1,\cdots,x_d) := h_{n_1}(x_1)\cdots h_{n_d}(x_d)$ for all $(x_1,\cdots,x_d) \in \R^d, n\in\mathbb{Z}^d_+$, where the Hermite functions on the right hand side are one-dimensional. We define the Hermite-Sobolev spaces $\Sc_p, p \in \R$ as the completion of $\Sc$ in
$\|\cdot\|_p$. Note that the dual space $\Sc_p^\prime$ is isometrically isomorphic with $\Sc_{-p}$ for $p\geq 0$. We also have $\Sc = \bigcap_{p}(\Sc_p,\|\cdot\|_p), \Sc^\prime=\bigcup_{p>0}(\Sc_{-p},\|\cdot\|_{-p})$ and $\Sc_0 = \mathcal{L}^2(\R^d)$.

For $x \in \R^d$, let $\tau_x$ denote the translation operators on $\Sc$ 
defined by
$(\tau_x\phi)(y):=\phi(y-x), \, \forall y \in \R^d$. These operators can be
extended to $\tau_x:\Sc'\to \Sc'$ by
\[\inpr{\tau_x\phi}{\psi}:=\inpr{\phi}{\tau_{-x}\psi},\, \forall \psi \in
\Sc.\]
\begin{proposition}\label{tau-x-estmte}
The translation operators $\tau_x, x \in \R^d$ have the following properties:
\begin{enumerate}[label=(\alph*)]
\item (\cite[Theorem 2.1]{MR1999259}) For $x \in \R^d$ and any $p \in \R$, $\tau_x: \Sc_p\to\Sc_p$
is a bounded linear map. In particular, there exists a real polynomial $P_k$ of
degree $k = 2(\lfloor|p|\rfloor +1)$ such that
\[\|\tau_x\phi\|_p\leq P_k(|x|)\|\phi\|_p, \, \forall \phi \in \Sc_p,\]
where $|x|$ denotes the Euclidean norm of $x$.
\item (\cite[Proposition 3.1]{MR2373102}) Fix $\phi \in \Sc_p$ for some $p \in \R$. The map $x \in\R^d \mapsto \tau_x\phi \in \Sc_p$ is continuous.
\end{enumerate}
\end{proposition}

\section{Finite dimensional SDEs}\label{S:3}
\subsection{setup and notations}\label{S:3-1}
We use the following notations throughout the paper.
\begin{itemize}
\item The set of positive integers will be denoted by $\N$.  Recall that for $x \in \R^n$, $|x|$ denotes its Euclidean norm. The transpose of any element $x \in \R^{n\times m}$ will be denoted by $x^t$.
\item For any $r > 0$, define $\oo(0,r):=\{x \in \R^d: |x|< r\}$. Then $\overline{\oo(0,r)} = \{x \in \R^d: |x| \leq r\}$ and $\oo(0,r)^c = \{x \in \R^d: |x|\geq r\}$.
\item Let $\big(\Omega,\F,\{\F_t\}_{t\geq0},P\big)$ be a filtered complete probability space satisfying the usual conditions viz. $\F_0$ contains all $A\in\F$, s.t. $P(A)=0$ and $\F_t=\bigcap_{s>t}\F_s, t \geq 0$.
\item Let $p>0$. Let $\sigma = (\sigma_{ij})_{d\times d}, b=(b_1, \cdots, b_d)^t$ be such that $\sigma_{ij}, b_i:\Omega\to\Sc_{p}$ are $\F_0$ measurable and 
\[\beta:= \sup\{\|\sigma_{ij}(\omega)\|_p, \|b_i(\omega)\|_p:\omega \in \Omega, 1 \leq i,j \leq d\} < \infty.\tag*{($\mathbf{\sigma b}$)} \label{sigma-b}\]
\item Define $\bar\sigma:\Omega\times\R^d\times\Sc_{-p}\to \R^{d\times d}$ and $\bar b:\Omega\times\R^d\times\Sc_{-p}\to \R^d$ by $\bar\sigma(\omega,z;y) := \inpr{\sigma(\omega)}{\tau_z y}$ and $\bar b(\omega,z;y) := 
\inpr{b(\omega)}{\tau_z y}$, where $(\inpr{\sigma(\omega)}{\tau_z y})_{ij}:= \inpr{\sigma_{ij}(\omega)}{\tau_z y}$ and $(\inpr{b(\omega)}{\tau_z y})_i := \inpr{b_i(\omega)}{\tau_z y}$.
\item Let $F:\Omega\times\mathcal{S}_{-p}\times \oo(0,1) \to 
\R^d$ and $G:\Omega\times\mathcal{S}_{-p}\times \oo(0,1)^c \to 
\R^d$ be $\F_0\otimes \mathcal{B}(\Sc_p)\otimes\mathcal{B}(\oo(0,1))/\mathcal{B}(\R^d)$ and $\F_0\otimes \mathcal{B}(\Sc_p)\otimes\mathcal{B}(\oo(0,1)^c)/\mathcal{B}(\R^d)$ measurable respectively. Here $\mathcal{B}(\K)$ denotes the Borel $\sigma$-field of set $\K$.
\item Define $\bar F: \Omega\times\R^d\times \oo(0,1)\times\Sc_{-p}\to \R^d$, $\bar G: \Omega\times\R^d \times \oo(0,1)^c\times\Sc_{-p}\to \R^d$ by $\bar F(\omega,z,x;y) := F(\omega,\tau_z y,x), \ \bar G(\omega,z,x;y) := G(\omega,\tau_z y, x)$.
\item Let $\{B_t\}$ denote a standard Brownian motion and let $N$ denote a Poisson random measure driven by a L\'evy measure $\nu$. $\widetilde N$ will denote the corresponding compensated random measure. We also assume that $B$ and $N$ are independent.
\end{itemize}

Consider the following SDE in $\R^d$, 
\begin{equation}\label{fd-sde-sln}
\begin{split}
dU_{t} &=\bar b(U_{t-};\xi)dt+ \bar\sigma(U_{t-};\xi)\cdot dB_t +\int_{(0 < |x| < 1)}  \bar F(U_{t-},x;\xi)\, \widetilde
N(dtdx)\\
&+\int_{(|x| \geq  1)} \bar G(U_{t-},x;\xi) \,
N(dtdx), \quad t\geq 0 \\
U_0&=\kappa,
\end{split}
\end{equation}
where $\xi$ is an $\Sc_{-p}$ valued $\F_0$-measurable random variable and $\kappa$ is an $\R^d$ valued $\F_0$-measurable random variable. Unless stated otherwise, $\xi$ and $\kappa$ will be taken to be independent of the noise $B$ and $N$. Note that the $i$-th component of $\int_0^t \bar\sigma(U_{s-};\xi)\cdot dB_s$ is $\sum_{j=1}^d\int_0^t \bar\sigma_{ij}(U_{s-};\xi)\, dB^j_s$. We list some hypotheses.

\begin{enumerate}[label=\textbf{(F\arabic*)},ref=\textbf{(F\arabic*)}]
\item\label{F1} For all $\omega\in\Omega$ and $x \in \oo(0,1)$ there exists a constant $C_x \geq 0$ s.t.
\begin{equation}\label{asm1}
\lvert F(\omega,y_1,x)-F(\omega,y_2,x)\rvert\leq C_x\|y_1-y_2\|_{-p-\frac{1}{2}}, \forall y_1,y_2\in\Sc_{-p}.
\end{equation}
We assume $C_x$ to depend only on $x$ and independent of $\omega$. Since $\|y\|_{-p-\frac{1}{2}} \leq \|y\|_{-p}, \forall y \in \Sc_{-p}$, we have 
\[\lvert F(\omega,y_1,x)-F(\omega,y_2,x)\rvert\leq C_x\|y_1-y_2\|_{-p}, \forall y_1,y_2\in\Sc_{-p}.\]
\item\label{F2} The constant $C_x$ mentioned above has the following properties, viz. 
\[\sup_{|x|<1}C_x<\infty,\quad \int_{(0<|x|<1)}C_x^2\, \nu(dx)<\infty.\]
\item\label{F3} $\sup_{\omega\in\Omega,|x|<1}|F(\omega,0,x)|<\infty$ and  $\sup_{\omega\in\Omega}\int_{(0<|x|<1)} |F(\omega,0,x)|^2\,\nu(dx)<\infty$. \end{enumerate}

\begin{enumerate}[label=\textbf{(G\arabic*)},ref=\textbf{(G\arabic*)}]
\item\label{G1} The mapping $y\rightarrow G(\omega,y,x)$ is continuous for all $x \in \oo(0,1)^c$ and $\omega \in \Omega$.
\end{enumerate}

\begin{remark}
Examples of coefficients $F$ and $G$ satisfying the above hypotheses can be constructed. See \cite[Example 3.1]{Levy-SPDE}.
\end{remark}

\begin{lemma}
[{\cite[Lemma 3.2]{Levy-SPDE}}]
\label{f-bd}
Assume \ref{F1}, \ref{F2} and \ref{F3}. Then, for any bounded set $\K$ in $\Sc_{-p}$ the following are true.
\begin{enumerate}[label=(\roman*)]
\item $\sup_{\omega\in\Omega,y\in \K,|x|<1}|F(\omega,y,x)|<\infty$.
\item $\sup_{\omega\in\Omega,y\in \K}\int_{(0<|x|<1)}|F(\omega,y,x)|^2\nu(dx)=:\alpha(\K)<\infty$.  
\item $\sup_{\omega\in\Omega,y\in \K}\int_0^t\int_{(0<|x|<1)}|F(\omega,y,x)|^4\nu(dx)ds<\infty$ for all $0\leq t<\infty$.
\end{enumerate}
\end{lemma}

Using the continuity result in Proposition \ref{tau-x-estmte} the next result follows.

\begin{lemma}[{\cite[Lemma 3.3]{Levy-SPDE}}]
Suppose \ref{G1} holds. Then the map $z \in \R^d \to \bar G(\omega, z,x; \xi(\omega)) = G(\omega,\tau_z\xi(\omega),x) \in \R^d$ is continuous for all $x \in \oo(0,1)^c$ and $\omega \in \Omega$.
\end{lemma}

\subsection{Global Lipschitz coefficients}\label{S:3-2}
In this subsection, we establish the existence and uniqueness of strong solutions of \eqref{fd-sde-sln} under `global Lipschitz' coefficients $\bar\sigma, \bar b, \bar F$. To do this we first study the same problem for the corresponding reduced equation, viz.
\begin{equation}\label{reduced-fd-sde}
\begin{split}
dU_{t} &=\bar b(U_{t-};\xi)dt+ \bar\sigma(U_{t-};\xi)\cdot dB_t +\int_{(0 < |x| < 1)}  \bar F(U_{t-},x;\xi)\, \widetilde
N(dtdx), \quad t\geq 0 \\
U_0&=\kappa;
\end{split}
\end{equation}
with $\xi$ and $\kappa$ as in \eqref{fd-sde-sln}. Later, in Theorem \ref{interlacing-global-sde} we prove the result for equation \eqref{fd-sde-sln}.

\begin{theorem}
\label{nrm-bd-rndm-inl}
Let \ref{sigma-b}, \ref{F1}, \ref{F2} and \ref{F3} hold. Suppose the following conditions are satisfied.
\begin{enumerate}[label=(\roman*)]
\item $\kappa, \xi$ are $\F_0$ measurable, as stated in \eqref{fd-sde-sln}. 
\item (Global Lipschitz in $z$, locally in $y$) For every bounded set $\K$ in $\Sc_{-p}$, there exists a constant $C(\K)>0$ such that for all $z_1, z_2\in\R^d,\ y\in \K$ and $\omega\in\Omega$
\begin{equation}\label{Lipschitz-condition-rnm-inl}
\begin{split}
&|\bar{b}(\omega,z_1;y) - \bar{b}(\omega,z_2;y)|^2+ |\bar{\sigma}(\omega,z_1;y)-
\bar{\sigma}(\omega,z_2;y)|^2\\
&+\int_{(0 < |x| < 1)}|\bar{F}(\omega,z_1,x;y) - \bar{F}(\omega,z_2,x;y)|^2 \, \nu(dx) \leq C(\K)\,
|z_1
- z_2|^2.
\end{split}
\end{equation}

\end{enumerate}
Then \eqref{reduced-fd-sde} has an $(\F_t)$ adapted strong solution $\{X_t\}$ with rcll paths. Pathwise uniqueness of solutions also holds, i.e. if $\{X^1_t\}$ is another such solution, then $P(X_t=X_t^1,t\geq0)=1$.
\end{theorem}
\begin{proof}
We split the proof in the following three steps, depending on assumptions on the random variables $\kappa$ and $\xi$.
\begin{enumerate}[label=Step \arabic*:]
\item $\kappa, \xi$ are $\F_0$ measurable with $\Exp|\kappa|^2<\infty$ and $\sup_{\omega\in\Omega} \|\xi(\omega)\|_{-p} < \infty$.
\item $\kappa, \xi$ are $\F_0$ measurable with $\Exp|\kappa|^2<\infty$.
\item $\kappa, \xi$ are $\F_0$ measurable.
\end{enumerate}
Positive constants appearing in our computations may be written as $\gamma$ and may change its values from line to line.

\underline{Step 1:} The existence is established by Picard iterations and the uniqueness by Gronwall inequality arguments. This follows the standard approach as in \cite[Theorem 5.2.1]{MR2001996}, where SDEs driven by Brownian motion were considered. In the present case, we get the linear growth of the coefficients directly from the structure of the coefficients, see \eqref{picard3} below.

First we prove the uniqueness. Let $\{U_t^1\}$ and $\{U_t^2\}$ be two solutions of \eqref{reduced-fd-sde}. Define, for $\omega\in\Omega$
\begin{align*}
&\Theta(t,\omega):=\bar b(\omega,U^1_{t-}(\omega);\xi(\omega))-\bar b(\omega,U^2_{t-}(\omega);\xi(\omega)),\\
&\Xi(t,\omega):=\bar\sigma(\omega,U^1_{t-}(\omega);\xi(\omega))-\bar\sigma(\omega,U^2_{t-}(\omega);\xi(\omega)),\\
&\Psi(t,x,\omega):=\bar F(\omega,U^1_{t-}(\omega),x;\xi(\omega))-\bar F(\omega,U^2_{t-}(\omega),x;\xi(\omega)).
\end{align*}
Using \eqref{Lipschitz-condition-rnm-inl}, Doob's $\mathcal{L}^2$ maximal
inequality and It\^o isometry, we have for some positive constant $\gamma$,
\begin{align}\label{diff1}
\begin{split}
&\Exp\left(\sup_{0\leq s\leq t}|U_s^1-U_s^2|^2\right)\\
&\leq 3t\ \Exp\int_0^t|\Theta(s)|^2ds+12\Exp\int_0^t|\Xi(s)|^2ds+12\Exp\int_0^t\int_{(0 < |x| < 1)}|\Psi(s,x)|^2\nu(dx)ds\\
&\leq 3\gamma(t+8) \int_0^t\Exp\left(\sup_{0\leq u\leq s}|U_{u}^1-U_{u}^2|^2\right)ds.
\end{split}
\end{align}
We then obtain the uniqueness of the solutions by a Gronwall inequality argument.

To show the existence of a strong solution, we use Picard iteration. Set $U_t^{(0)} = \kappa$ and define
\begin{equation}\label{fd-sde-sln-rndm-inl-pkrd}
U_{t}^{(k+1)}:=\kappa+\int_0^t\bar b(U_{s-}^{(k)};\xi)ds+ \int_0^t\bar\sigma(U_{s-}^{(k)};\xi)\cdot dB_s
+\int_0^t\int_{(0 < |x| < 1)}  \bar F(U_{s-}^{(k)},x;\xi)\, \widetilde
N(dsdx),
\end{equation}
for all $k\geq0$. Fix $M \in \N$. For $k\geq1$, $t\in[0,M]$ we have
\begin{equation}\label{picard1}
\Exp\left(\sup_{0\leq s\leq t}|U_s^{(k+1)}-U_s^{(k)}|^2\right)\leq 3\gamma(M+8)\int_0^t\Exp\left(\sup_{0\leq u\leq s}|U_u^{(k)}-U_u^{(k-1)}|^2\right)ds.
\end{equation}
By \eqref{Lipschitz-condition-rnm-inl}, there exists a constant $C = C(Range(\xi))$ such that for $z\in\R^d$, $y\in Range(\xi)$
\begin{equation}\label{picard2}
\begin{split}
&|\bar{b}(\omega,z;y) - \bar{b}(\omega,0;y)|^2+ |\bar{\sigma}(\omega,z;y)-
\bar{\sigma}(\omega,0;y)|^2\\
&+\int_{(0 < |x| < 1)}|\bar{F}(\omega,z,x;y) - \bar{F}(\omega,0,x;y)|^2 \, \nu(dx) \leq C\,
|z|^2.
\end{split}
\end{equation}

Using \ref{sigma-b}), we have $|\bar b(\omega,0;y)|=|\langle b(\omega),y\rangle|\leq \beta \sqrt{d}\|y\|_{-p}$ and $|\bar\sigma(\omega,0;y)|=|\langle \sigma(\omega),y\rangle|\leq \beta d \|y\|_{-p}$. From \ref{F1}, we have $|\bar F(\omega,0,x;y)|=|F(\omega,y,x)|\leq C_x\|y\|_{-p}+|F(\omega,0,x)|$.

Therefore, using \eqref{picard2}, \ref{F2} and \ref{F3}, there exists a constant $D=D(Range(\xi))>0$ such that
\begin{equation}\label{picard3}
|\bar{b}(\omega,z;y)|^2+ |\bar{\sigma}(\omega,z;y) |^2
+\int_{(0 < |x| < 1)}|\bar{F}(\omega,z,x;y)  |^2 \, \nu(dx) \leq D\,
(1+|z|^2).
\end{equation}
As in \eqref{diff1}, using \eqref{fd-sde-sln-rndm-inl-pkrd}, Doob's $\mathcal{L}^2$ maximal
inequality and It\^o isometry and \eqref{picard3} we get  
\begin{equation}\label{picard4}
\Exp\left(\sup_{0\leq s\leq t}|U_s^{(1)}-U_s^{(0)}|^2\right) \leq(3t^2+24t) D \ \Exp (1+|\kappa|^2).
\end{equation}
Therefore by induction from \eqref{picard1}, there exists a positive constant $\tilde C$ s.t.
\begin{align}\label{picard5}
\Exp\left(\sup_{0\leq s\leq t}|U_s^{(k+1)}-U_s^{(k)}|^2\right)\leq\frac{(\tilde Ct)^{k+1}}{(k+1)!},\ \forall k\geq0,\ t\in[0,M].
\end{align} 
For positive integers $m,\ n$ with $m>n$, we have 
\begin{align}\label{converge-123}
\begin{split}
\lim_{m,n\rightarrow\infty}\Exp\sup_{0\leq t\leq M}|U_t^{(m)}-U_t^{(n)}|^2 &=\lim_{m,n\rightarrow\infty}\Exp\sup_{0\leq t\leq M}\left|\sum_{k=n}^{m-1}\big(U_t^{(k+1)}-U_t^{(k)}\big)\right|^2\\
&\leq\lim_{n\rightarrow\infty}\sum_{k=n}^{\infty}\Exp\sup_{0\leq t\leq M}\big|U_t^{(k+1)}-U_t^{(k)}\big|^2k^2\left(\sum_{k=n}^{\infty}k^{-2}\right).
\end{split}
\end{align}
The second series on the right hand side above converges. By \eqref{picard5}, the first series is bounded, since $\sum_{k=n}^{\infty}\frac{(\tilde CM)^{k+1}}{(k+1)!}k^2\rightarrow0$ as $n\rightarrow\infty$. Therefore   $\{U_t^{(m)}:m\in\N\}$ is Cauchy and hence converges to some $\{X_t\}_{t\in[0,M]}$ in $\mathcal L^2(\lambda\times P)$, where $\lambda$ denotes the Lebesgue measure on $[0,M]$.\\
Applying the Chebyshev-Markov inequality in \eqref{picard5}, we get
\begin{align*}
P\left(\sup_{0\leq s\leq t}|U_s^{(k+1)}-U_s^{(k)}|\geq\frac{1}{2^{k+1}}\right)\leq\frac{(4\tilde Ct)^{k+1}}{(k+1)!}.
\end{align*}
By Borel-Cantelli lemma
\begin{align*}
P\left(\limsup_{k\rightarrow\infty}\sup_{0\leq s\leq t}|U_s^{(k+1)}-U_s^{(k)}|\geq\frac{1}{2^{k+1}}\right)=0.
\end{align*}
Therefore, we conclude that $\{U^{(k)}\}$ is almost surely uniformly convergent on $[0,M]$ to $\{X_t\}$, which is adapted and rcll. Using \eqref{picard3} and the fact that a.s. $\{X_t\}$ has at most countably many jumps, we have
\[\Exp\int_0^M\int_{(0 < |x| < 1)}|\bar F(X_{s-},x;\xi)|^2\nu(dx)ds \leq\Exp\int_0^M D(1+|X_{s-}|^2)ds \leq D \left[ M + \|X\|^2_{\mathcal L^2(\lambda\times P)}\right] < \infty.\]
Therefore $\{\int_0^t\int_{(0 < |x| < 1)}  \bar F(X_{s-},x;\xi)\, \widetilde
N(dsdx)\}_{t\in[0,M]}$ exists. Similarly, we can show the existence of $\{\int_0^t\bar\sigma(X_{s-};\xi)\cdot dB_s\}_{t\in[0,M]}$ and $\{\int_0^t\bar b(X_{s-};\xi)ds\}_{t\in[0,M]}$.

By It\^o isometry and \eqref{Lipschitz-condition-rnm-inl}, we have the following convergence in $\mathcal L^2(P)$, viz. \[\int_0^t\int_{(0 < |x| < 1)}  \bar F(U_{s-}^{(k)},x;\xi)\, \widetilde
N(dsdx) \xrightarrow{k\to \infty} \int_0^t\int_{(0 < |x| < 1)}  \bar F(X_{s-},x;\xi)\, \widetilde
N(dsdx),\]  
for each $t\in[0,M]$. Similarly, we conclude that $\int_0^t\bar\sigma(U_{s-}^{(k)};\xi)\cdot dB_s\rightarrow\int_0^t\bar\sigma (X_{s-};\xi)\cdot dB_s$ and $\int_0^t\bar b(U_{s-}^{(k)};\xi) ds\rightarrow\int_0^t\bar b(X_{s-};\xi)ds$ in $\mathcal L^2(P)$ as $k\rightarrow\infty$, for each $t\in[0,M]$. Since $\{X_t\}$ is rcll, from \eqref{fd-sde-sln-rndm-inl-pkrd}, we have a.s. $\forall t\in[0,M]$,
\[
X_t=\kappa+\int_0^t\bar b(X_{s-};\xi)ds+ \int_0^t\bar\sigma(X_{s-};\xi)\cdot dB_s
+\int_0^t\int_{(0 < |x| < 1)}  \bar F(X_{s-},x;\xi)\, \widetilde
N(dsdx).
\]
Suppose $\{X_t^{(M)}\}$ and $\{X_t^{(M+1)}\}$ denote the solutions up to time $M$ and $M+1$ respectively. Then, by the uniqueness, $\{X_t^{(M+1)}\}_{t\in[0,M]}$ is indistinguishable from $\{X_t^{(M)}\}$ on $[0,M]$. Using this consistency, we obtain the solution of \eqref{reduced-fd-sde} on the time interval $[0,\infty)$. This concludes the proof for Step 1.

\underline{Step 2:} We follow the technique given in \cite[Theorem 3.3]{MR2560625}, where SDEs driven by Brownian motion were considered. For $k \in \N$, define $\chi_k:=\mathbbm1_{\{\|\xi\|_{-p}\leq k\}}$ and let $\xi^{(k)}:=\chi_k\ \xi$. Let $U^{(k)}$ be the solution of \eqref{reduced-fd-sde} with the initial condition $\xi^{(k)}$. Our aim is to show that $\chi_kU^{(k)}=\chi_kU^{(k+1)}$. Let $U^{(k)}_n$ and $U^{(k+1)}_n$ be the approximations of $U^{(k)}$ and $U^{(k+1)}$ obtained in Step 1 above. Now,
\[U^{(k)}_0(t)=\kappa,\ U^{(k+1)}_0(t)=\kappa\ \text{and}\  \chi_kU^{(k)}_0(t)=\chi_kU^{(k+1)}_0(t).\]
Observe that, for $\omega\in\Omega$
\[\chi_k(\omega)\bar b(\omega,U^{(k)}_0(s-)(\omega);\xi^{(k)}(\omega))=\chi_k(\omega)\bar b(\omega,U^{(k+1)}_0(s-)(\omega);\xi^{(k+1)}(\omega)).\] 
Similar equalities hold for coefficients $\bar \sigma$ and $\bar F$.
Using \eqref{fd-sde-sln-rndm-inl-pkrd} and these equalities, a.s. $t\geq 0$, $\chi_kU_{1}^{(k)}(t)=\chi_kU_{1}^{(k+1)}(t)$. By induction a.s. $t\geq 0, \chi_kU_{n}^{(k)}(t)=\chi_kU_{n}^{(k+1)}(t)$.

Letting $n$ go to infinity and using the generalized Lebesgue DCT (see   \cite[Theorem 3.4]{MR2560625}), we have, a.s. $\forall t\in[0,T]$, $\chi_kU^{(k)}(t)=\chi_kU^{(k+1)}(t)$. Note that $P\big(\bigcup_k\{\chi_k=1\}\big)=1$. Now define
\[X_t(\omega):=U^{(k)}(t)(\omega),\ \ \text{if $\|\xi(\omega)\|_{-p}\leq k$.}\]
Observe that, a.s. $\forall t\in[0,T], \chi_k U^{(k)}(t) = \chi_k X_t$. It is easy to check that $\{X_t\}$ satisfies \eqref{reduced-fd-sde}.

To prove the uniqueness, let $\{X_t\}$ and $\{Y_t\}$ be two solutions of \eqref{reduced-fd-sde}. Define
\[\widetilde F(\omega,z,x;y):=\mathbbm1_{\{\tilde y:\|\tilde y\|_{-p}\leq k\}}(y)\bar F(\omega,z,x;\mathbbm1_{\{\tilde y:\|\tilde y\|_{-p}\leq k\}}(y)y),\]
and $X^k_t:=\chi_k X_t$, for $\omega\in\Omega, k \in \N$. Similarly define $\{Y^k_t\}$ for $k \in \N$. Observe that
\begin{align*}
\widetilde F(\omega,z,x;\xi(\omega))&=\mathbbm1_{\{\tilde y:\|\tilde y\|_{-p}\leq k\}}(\xi(\omega))\bar F\big(\omega,z,x;\mathbbm1_{\{\tilde y:\|\tilde y\|_{-p}\leq k\}}(\xi(\omega))\xi(\omega)\big)\\
&=\mathbbm1_{\{\tilde\omega:\|\xi(\tilde\omega)\|_{-p}\leq k\}}(\omega)\bar F\big(\omega,z,x;\mathbbm1_{\{\tilde\omega:\|\xi(\tilde\omega)\|_{-p}\leq k\}}(\omega)\xi(\omega)\big),
\end{align*}
and
\begin{align*}
&\chi_k(\omega)\bar b(\omega,X_{s-}(\omega);\xi(\omega)) =\bar b(\omega,X^k_{s-}(\omega);\xi^k(\omega)),\\
&\chi_k(\omega)\bar\sigma(\omega,X_{s-}(\omega);\xi(\omega))=\bar\sigma(\omega,X^k_{s-}(\omega);\xi^k(\omega)),\\
&\chi_k(\omega)\bar F(\omega,X_{s-}(\omega),x;\xi(\omega)) =\chi_k(\omega)\bar  F(\omega,X^k_{s-}(\omega),x;\chi_k(\omega)\xi^k(\omega))=\widetilde  F(\omega,X^k_{s-}(\omega),x;\xi^k(\omega)).
\end{align*}
Therefore,
\begin{equation}\label{x^k-unq}
\begin{split}
X_t^k &= \chi_k X_t\\ 
&= \chi_k\ \kappa + \int_0^t \bar b(X^k_{s-};\xi^k)ds+ \int_0^t\bar\sigma(X^k_{s-};\xi^k)\cdot dB_s
+\int_0^t\int_{(0 < |x| < 1)}\widetilde F(X^k_{s-},x;\xi^k)\, \widetilde
N(dsdx).
\end{split}
\end{equation}
Now, in \eqref{x^k-unq} $\xi^k$ is norm bounded. Moreover, it is easy to check that $\bar b$, $\bar\sigma$ and $\widetilde F$ satisfy \eqref{Lipschitz-condition-rnm-inl}. By the uniqueness in Step 1, we conclude that $\{X^k_t\}$ is the unique solution of \eqref{reduced-fd-sde} with initial condition $\chi_k\ \kappa$ and in particular,
\[\chi_k(\omega)X_t=X^k_t=Y^k_t=\chi_k(\omega)Y_t.\]
Since $k$ is arbitrary, therefore, a.s. $\forall t\in[0,T]$, $X_t=Y_t$. This completes the proof for Step 2.

\underline{Step 3:} We follow the argument given in  \cite[Theorem 6.2.3]{MR2512800}. Define $\Omega_M:=\{\omega\in\Omega:\ |\kappa|\leq M\}$ for each $M\in\N$. Then $\Omega=\bigcup_{M\in\N}\Omega_M$ and $\Omega_L\subseteq\Omega_M$ whenever $L\leq M$. 

Let $\kappa^M(\omega):=\mathbbm1_{\{|\kappa|\leq M\}}(\omega)\kappa(\omega)$. Note that $\kappa^M\in\mathcal L^2$. By Step 2, there exists a unique solution, say $\{X_t^{\kappa^M}\}$, of the reduced equation \eqref{reduced-fd-sde} for the initial condition $\kappa^M$, i.e. a.s. $t \geq 0$ 
\[X_t^{\kappa^M} = \kappa^M + \int_0^t\bar b(X^{\kappa^M}_{s-};\xi)ds +  \int_0^t\bar\sigma(X^{\kappa^M}_{s-};\xi)\cdot dB_s + \int_0^t\int_{(0 < |x| < 1)} \bar F(X^{\kappa^M}_{s-},x;\xi)\, \widetilde
N(dsdx).\]
We first show a.s. $\mathbbm1_{\{|\kappa|\leq L\}}(\omega) X_t^{\kappa^L}(\omega) = \mathbbm1_{\{|\kappa|\leq L\}}(\omega) X_t^{\kappa^M}(\omega), t \geq 0$ for all $M\geq L$. Define
\[\widetilde{F}(\omega,z,x;y):=\mathbbm1_{\{|\kappa|\leq L\}}(\omega)\bar F(\omega,z,x;y).\]
Now, $\{\mathbbm1_{\{|\kappa|\leq L\}} X_t^{\kappa^L}\}$ and $\{\mathbbm1_{\{|\kappa|\leq L\}} X_t^{\kappa^M}\}$ both satisfy the reduced equation
\begin{equation}
\begin{split}
dX_{t} &=\bar b(X_{t-};\mathbbm1_{\{|\kappa|\leq L\}}\xi)dt+ \bar\sigma(X_{t-};\mathbbm1_{\{|\kappa|\leq L\}}\xi)\cdot dB_t +\int_{(0 < |x| < 1)}  \widetilde F(X_{t-},x;\mathbbm1_{\{|\kappa|\leq L\}}\xi)\, \widetilde
N(dtdx),\\
X_0&=\kappa^L.
\end{split}
\end{equation}
It is easy to check that $\bar b$, $\bar\sigma$, $\widetilde F$ satisfy \eqref{Lipschitz-condition-rnm-inl}. Then by the uniqueness in Step 2 for all $M\geq L$ a.s. \[\mathbbm1_{\{|\kappa|\leq L\}} X_t^{\kappa^L} = \mathbbm1_{\{|\kappa|\leq L\}} X_t^{\kappa^M},\ t \geq 0.\]
Since $\Omega_M$ increases to $\Omega$, for all $\epsilon>0$, there exists $M\in\N$, such that $P(\Omega_n)>1-\epsilon, \forall n > M$. Hence, 
\[P\left(\sup_{t\geq0}|X^{\kappa^m}_t-X^{\kappa^n}_t|>\delta\right)<\epsilon, \ \forall \delta > 0, \forall m,n>M.\]
Therefore the sequence of processes $\{X^{\kappa^n}\}_{n\in\N}$ is uniformly Cauchy in probability and so is uniformly convergent in probability to a process, say $\{X_t\}$. We extract a subsequence for which the convergence holds uniformly and almost surely. This convergence implies that $\{X_t\}$ has rcll paths and solves \eqref{reduced-fd-sde}.

To prove the uniqueness, we consider the solution $\{X_t\}$ constructed above and compare it with any arbitrary solution $\{X'_t\}_{t\geq0}$ of \eqref{reduced-fd-sde}. We claim that for all $M\geq L$, $X'_t(\omega)=X^{\kappa^M}_t(\omega)$ for all $t\geq0$ and almost all $\omega\in\Omega_L$. Suppose for some $M\geq L$, it doesn't hold. Define
\[X''^{\kappa^M}_t(\omega):=
\begin{cases}
X'_t(\omega) \ \ \text{for}\ \omega\in\Omega_L,\\
X^{\kappa^M}_t(\omega),\ \text{for}\ \omega\in\Omega_L^c.
\end{cases}\]
Then $X''^{\kappa^M}$ and $X^{\kappa^M}$ are two distinct solutions of  \eqref{reduced-fd-sde} with the same initial condition $\kappa^M$, which is a contradiction. This proves our claim. Next by applying a limiting argument we conclude that $P(X_t=X'_t, \forall t\geq0)=1$. This completes the proof of Step 3 as well as the theorem.
\end{proof}

We now consider the SDE \eqref{fd-sde-sln}. The next result follows by the interlacing technique (see \cite[Example 1.3.13, pp. 50-51]{MR2512800}).

\begin{theorem}\label{interlacing-global-sde}
Suppose all the assumptions of Theorem \ref{nrm-bd-rndm-inl} hold. In addition, assume that \ref{G1} holds. Then there exists a unique rcll adapted solution to \eqref{fd-sde-sln}.
\end{theorem}

\begin{proof}
We follow the proof of \cite[Theorem 6.2.9]{MR2512800}. We have already proved the existence and uniqueness of the reduced equation in Theorem \ref{nrm-bd-rndm-inl}. Now, we use the interlacing technique to complete the proof.

Let $\{\eta_n\}_{n\in\N}$ denote the arrival times for the jumps of the compound Poisson process $\{P_t\}_{t\geq0}$, where each $P_t=\int_{(|x|\geq1)}xN(t,dx)$. By Theorem \ref{nrm-bd-rndm-inl} there exists a unique solution $\{\widetilde U^{(1)}_t\}$ to the reduced equation \eqref{reduced-fd-sde}. Define

\[U_t := \begin{cases}
         \widetilde U_t^{(1)};\ \ \ \text{for $0\leq t<\eta_1$}\\
         \widetilde U_{\eta_1-}^{(1)} + \bar G(\widetilde U_{\eta_1-}^{(1)},\triangle P_{\eta_1};\xi);\ \ \ \text{for $t=\eta_1$}\\
         U_{\eta_1} + \widetilde U^{(2)}_{t}-\widetilde U^{(2)}_{\eta_1};\ \ \ \text{for $\eta_1< t<\eta_2$}\\
         U_{\eta_2-}+\bar G( U_{\eta_2-},\triangle P_{\eta_2};\xi);\ \ \ \text{for $t=\eta_2$}\\
         \cdots
         \end{cases}
\]
Here $\{\widetilde U^{(2)}_t\}$ denotes the unique solution to \eqref{reduced-fd-sde} with initial condition $U_{\eta_1}$. Then $\{U_t\}$ is an adapted rcll process and solves \eqref{fd-sde-sln}.

We show that the uniqueness follows by the interlacing structure. Let $\{\hat{U}_t\}$ be another solution of \eqref{fd-sde-sln}. Then by the uniqueness of the reduced equation, a.s.
\[
\hat U_t=\widetilde U_t=U_t;\ \ \text{for $0\leq t<\eta_1$}.\]
Since, a.s. $\hat U_{\eta_1-}=\widetilde U_{\eta_1-} = U_{\eta_1-}$, we have a.s.
\[\hat U_{\eta_1} = \hat  U_{\eta_1-} + \bar G(\hat  U_{\eta_1-},\triangle P_{\eta_1};\xi) = \widetilde U_{\eta_1-} + \bar G(\widetilde U_{\eta_1-},\triangle P_{\eta_1};\xi) = U_{\eta_1}.
\]
Since $\{\hat U_t\}$ has no large jump in the time interval $(\eta_1, \eta_2)$ we have, a.s. for $t \in (\eta_1, \eta_2)$
\begin{equation}\label{interlace-sde-eta-12}
\begin{split}
\hat U_t &= \hat U_{\eta_1}+\int_{\eta_1}^t\bar b(\hat U_{s-};\xi)ds+\int_{\eta_1}^t \bar\sigma(\hat U_{s-};\xi)\cdot dB_s +\int_{\eta_1}^t\int_{(0 < |x| < 1)}  \bar F(\hat U_{s-},x;\xi)\, \widetilde
N(dsdx)\\
&=\hat U_{\eta_1}+\int_0^{t-\eta_1}\bar b(\hat U_{\eta_1+s-};\xi)ds+\int_0^{t-\eta_1} \bar\sigma(\hat U_{\eta_1+s-};\xi)\cdot dB_{\eta_1+s}\\ &\ \ +\int_0^{t-\eta_1}\int_{(0 < |x| < 1)}  \bar F(\hat U_{\eta_1+s-},x;\xi)\, \widetilde
N_s^{\eta_1}(dsdx). 
\end{split}
\end{equation}
We now describe $\{N_s^{\eta_1}\}$, which appeared in the last term of \eqref{interlace-sde-eta-12}. For any set $H \subset \R^d$, which is bounded away from $0$, i.e. $0 \notin \bar H$ and for any stopping time $\eta$, define \[N^{\eta}_{t}(H):=\left(N_{t+\eta}(H)-N_{\eta}(H)\right) \mathbbm{1}_{(\eta<\infty)}.\]
By strong Markov property \cite[Theorem 2.2.11]{MR2512800}, we have
$\Exp[e^{i\lambda N^{\eta}_{t}(H)}]=\Exp[e^{i\lambda N_{t}(H)}]$, $\{N^{\eta}_{t}\}$ is independent of $\F_{\eta}$, has rcll paths and is $(\F_{\eta+t})$ adapted. Furthermore,
$\Exp[N^{\eta}_{t}(H)]=t\nu(H)=\Exp[N_t(H)]$.

Note that the last equality of \eqref{interlace-sde-eta-12} is written in the reduced equation form. Since $\{U_t\}$ also solves the same reduced equation, by Theorem \ref{nrm-bd-rndm-inl} a.s. $\hat U_t=U_t$ for $\eta_1<t<\eta_2$. In particular, a.s. $\hat U_{\eta_2-} = U_{\eta_2-}$ and hence, a.s.
\[\hat U_{\eta_2} = \hat  U_{\eta_2-} + \bar G(\hat  U_{\eta_2-},\triangle P_{\eta_2};\xi) = U_{\eta_2-}+\bar G( U_{\eta_2-},\triangle P_{\eta_2};\xi) = U_{\eta_2}.\]
Continuing this way, we show that a.s. $U_t = \hat U_t, t \geq 0$. This completes the proof.
\end{proof}

\subsection{Local Lipschitz coefficients}\label{S:3-3}
In the previous subsection, we have established the existence and uniqueness results under `global Lipschitz' which we now extend for `local Lipschitz' coefficients.

Let $\widehat{\R^d}:=\R^d\cup\{\infty\}$ be the one point compactification of $\R^d$.  

\begin{theorem}
\label{nrm-sqre-rndm-inl-fnl}
Let \ref{sigma-b}, \ref{F1}, \ref{F2}, \ref{F3} and \ref{G1} hold. Suppose the following conditions are satisfied.
\begin{enumerate}[label=(\roman*)]
\item $\kappa, \xi$ are $\F_0$-measurable.
\item (Locally Lipschitz in $z$, locally in $y$) For every bounded set $\K$ in $\Sc_{-p}$ and positive integer $n$ there exists a constant $C(\K,n)>0$ s.t. for all $z_1, z_2\in \oo(0,n),\ y\in \K$ and $\omega\in\Omega$
\begin{equation}\label{Lipschitz-condition-rnm-inl-nrmsqr-fnl}
\begin{split}
&|\bar{b}(\omega,z_1;y) - \bar{b}(\omega,z_2;y)|^2+ |\bar{\sigma}(\omega,z_1;y)-
\bar{\sigma}(\omega,z_2;y)|^2\\
&+\int_{(0 < |x| < 1)}|\bar{F}(\omega,z_1,x;y) - \bar{F}(\omega,z_2,x;y)|^2 \, \nu(dx) \leq C(\K,n)\,
|z_1
- z_2|^2.
\end{split}
\end{equation}
\end{enumerate}
Then there exists an $(\F_t)$ stopping time $\eta$ and an $(\F_t)$ adapted $\widehat{\R^d}$ valued process $\{X_t\}$ with rcll paths such that $\{X_t\}$ solves \eqref{fd-sde-sln} upto time $\eta$ and $X_t=\infty$ for $t\geq\eta$. Further $\eta$ can be identified as follows: $\eta=\lim_m\theta_m$ where $\{\theta_m\}$ are $(\F_t)$ stopping times defined by $\theta_m:=\inf\{t\geq0:|X_t|\geq m\}$.\\
This is also pathwise unique in this sense: if $(\{X'_t\},\eta')$ is another such solution, then $P(X_t=X'_t, 0\leq t<\eta\wedge\eta')=1$.
\end{theorem}

\begin{proof}
To prove the existence result, we first obtain a version of the `global Lipschitz' condition \eqref{Lipschitz-condition-rnm-inl} for $\bar b(\omega,z;y)$, $\bar\sigma(\omega,z;y)$, $\bar F(\omega,z,x;y)$ from our assumption on `local Lipschitz' condition  \eqref{Lipschitz-condition-rnm-inl-nrmsqr-fnl}.\\
Let $n,m \in \N$ and let $R$ be a positive real number. Let $h:\R^n \to \R^m$ satisfy $|h(x)-h(y)|\leq C|x-y|$ for all $x,y$ with $|x|, |y|\leq R$, where $C$ is a positive constant. Define
\[h^R(x) := \begin{cases}
            h(x), \ \text{if}\, |x| \leq R\\
            \frac{2R-|x|}{R}\cdot h(Rx/|x|),\ \text{if}\ R\leq|x|\leq 2R\\
            0, \ \text{if}\ |x| \geq	 2R.
            \end{cases}
\]
By \cite[Chapter 5, Exercise 3.1]{MR1398879}, $h^R$ is Lipschitz continuous on $\R^n$. For every fixed $y$ and $\omega$, we construct $\bar\sigma^R(\omega,\cdot;y)$ for $\bar\sigma(\omega,\cdot;y)$ in the same way viz.,
\[\bar\sigma^R(\omega,z;y):= \begin{cases}
\bar\sigma(\omega,z;y),\ \text{for}\ |z|\leq R;\\
\frac{2R-|z|}{R}\cdot\bar\sigma\left(\omega,\frac{Rz}{|z|};y\right),\ \text{for}\ R\leq|z|\leq2R,\\
0,\ \text{for}\ |z|\geq2R.
\end{cases}
\]
Similarly define $\bar b^R(\omega,\cdot;y)$ and $\bar F^R(\omega,\cdot,x;y)$ for every fixed $x,y$ and $\omega$. Then using \eqref{Lipschitz-condition-rnm-inl-nrmsqr-fnl} and applying the above exercise, we conclude that $\bar b^R(\omega,z;y)$ and $\bar\sigma^R(\omega,z;y)$ are globally Lipschitz in $z$ as in \eqref{Lipschitz-condition-rnm-inl}. We now show \eqref{Lipschitz-condition-rnm-inl} holds for $\bar F^R(\omega,z,x;y)$.\\
By \eqref{Lipschitz-condition-rnm-inl-nrmsqr-fnl} and Lemma \ref{f-bd}, for any $z \in \R^d$ with $|z| \leq R$ and any bounded set $\K$ in $\Sc_{-p}$, we have 
\begin{equation}\label{observe}
\begin{split}
&\int_{(0<|x|<1)}\left|\bar F\left(\omega,z,x;y\right)\right|^2\nu(dx)\\
&\leq 2 \int_{(0<|x|<1)}\left|\bar F\left(\omega,z,x;y\right)-\bar F\left(\omega,0,x;y\right)\right|^2\nu(dx) + 2\int_{(0<|x|<1)}\left|\bar F\left(\omega,0,x;y\right)\right|^2\nu(dx)\\
&\leq 2 C(\K,R)R^2 +2 \alpha(\K), \ \forall y \in \K.
\end{split}
\end{equation}

Fix $z_1, z_2 \in \R^d$, with
$|z_1|\leq R$ and $R\leq|z_2|\leq 2R$. Then
\begin{align*}
&\int_{(0<|x|<1)}|\bar F^R(\omega,z_1,x;y)-\bar F^R(\omega,z_2,x;y)|^2\nu(dx)\\
&=\int_{(0<|x|<1)}\left|\bar F(\omega,z_1,x;y)-\frac{2R-|z_2|}{R}\cdot\bar F\left(\omega,\frac{Rz_2}{|z_2|},x;y\right)\right|^2\nu(dx)\\
&\leq2\int_{(0<|x|<1)}\left|\bar F(\omega,z_1,x;y)-\bar F\left(\omega,\frac{Rz_2}{|z_2|},x;y\right)\right|^2\nu(dx)\\
&\ \ \ \ \ \ +2\frac{||z_2|-R|^2}{R^2}\int_{(0<|x|<1)}\left|\bar F\left(\omega,\frac{Rz_2}{|z_2|},x;y\right)\right|^2\nu(dx)\\
&\leq2C(\K,R)\left|z_1-\frac{Rz_2}{|z_2|}\right|^2+2\frac{||z_2|-R|^2}{R^2} \left[2 C(\K,R)R^2 +2 \alpha(\K) \right]\\ 
&=|z_1-z_2|^2\left[6C(\K,R)+\frac{4}{R^2}\alpha(\K)\right].
\end{align*}
In the above calculation, we have used \eqref{observe} and two inequalities, viz. $\left|z_1-\frac{Rz_2}{|z_2|}\right|^2\leq|z_1-z_2|^2$ and $||z_2|-R|^2\leq|z_1-z_2|^2$. These inequalities are easy to verify. For example,  the first one follows from the equivalent statement $|z_1|^2+R^2-2\frac{R}{|z_2|} (z_1)^t z_2 \leq|z_1|^2+|z_2|^2 - 2(z_1)^t z_2$.

Similar arguments show that \eqref{Lipschitz-condition-rnm-inl} holds for $\bar F^R$ for all $z_1, z_2 \in \R^d$. This shows that the `global Lipschitz' regularity \eqref{Lipschitz-condition-rnm-inl} holds for $\bar b^R$, $\bar\sigma^R$ and $\bar F^R$. Since $\bar b^R(\omega,0;y) = \bar b(\omega,0;y), \bar \sigma^R(\omega,0;y) = \bar \sigma(\omega,0;y)$ and $\bar F^R(\omega,0,x;y) = \bar F(\omega,0,x;y), \forall |x| < 1, y \in \Sc_{-p}$, the growth condition \eqref{picard3} can be established for $\bar b^R$, $\bar\sigma^R$ and $\bar F^R$ as done in Step 1 of Theorem \ref{nrm-bd-rndm-inl}. Then arguing as in Theorem \ref{nrm-bd-rndm-inl} (Steps 1, 2 and 3) and Theorem \ref{interlacing-global-sde}, for $R \in \N$, we have the existence of a unique process $\{X^R_t\}$ satisfying a.s. for every $t\geq0$
\begin{equation}
\begin{split}
X^R_t &=\kappa+\int_0^t\bar b^R(X_{s-}^R;\xi)ds+ \int_0^t\bar\sigma^R(X_{s-}^R;\xi)\cdot dB_s
+\int_0^t\int_{(0 < |x| < 1)}  \bar F^R(X_{s-}^R,x;\xi)\, \widetilde
N(dsdx)\\
& +\int_0^t \int_{(|x| \geq  1)} \bar G(X_{s-}^R,x;\xi) \,
N(dsdx).
\end{split}
\end{equation}
Let $\pi_i, i=1,2,\cdots$ denote the arrival times for the jumps of the compound Poisson process $\{P_t\}_{t\geq0}$, where each $P_t=\int_{(|x|\geq1)}xN(t,dx)$. Let $m,n\in\N$ and $m<n$. Consider the stopping times
\[\theta_{m,i}^n:= \inf\{t\geq0:\ |X^m_t|\ \text{Or}\ |X^n_t|\geq m\}\wedge \pi_i.\]
Take $i = 1$. Then $\{X^m_t\}$ and $\{X^n_t\}$ both satisfy the same reduced equation
\begin{equation}
\begin{split}
dX_{t} &=\bar b^m(X_{t-};\xi)dt+ \bar\sigma^m(X_{t-};\xi)\cdot dB_t + \int_{(0 < |x| < 1)}  \bar F^m(X_{t-},x;\xi)\, \widetilde
N(dtdx), \quad t < \theta_{m,1}^n,\\
X_0&=\kappa;
\end{split}
\end{equation}
First assume $\xi$ is norm bounded and consider the stopped processes $\{X^m_{t\wedge \theta_{m,1}^n}\}$ and $\{X^n_{t\wedge \theta_{m,1}^n}\}$. Then arguing as in the uniqueness proof of Step 1 in Theorem \ref{nrm-bd-rndm-inl}, we conclude a.s. $X^m_t = X^n_t, t < \theta_{m,1}^n$. In particular, a.s. $X^m_{t-} = X^n_{t-}$ for $t = \theta_{m,1}^n$. Further, for almost all $\omega$ such that $\pi_1(\omega) = \theta_{m,1}^n(\omega)$, we have
\[X^m_t(\omega) = X^m_{t-}(\omega) + \bar G(X^m_{t-}(\omega),\triangle N_t,;\xi) = X^n_t(\omega),\quad t = \pi_1(\omega).\]
We extend this result for $\F_0$ measurable $\xi$ by arguing as in Step 2 in Theorem \ref{nrm-bd-rndm-inl}.

Take $i=2$. Note that the contribution of the term involving $\bar G$ in $X^m_{t\wedge \theta_{m,2}^n}$ and $X^n_{t\wedge \theta_{m,2}^n}$ for the large jump at $t = \pi_1$ are the same. Arguing as in the case $i = 1$, we conclude a.s $X^m_t = X^n_t, t < \theta_{m,2}^n$. 

Repeating the arguments, we have a.s. for all $i, m, n$ with $m < n, X^m_t = X^n_t, t < \theta_{m,i}^n$. Since a.s. $\pi_i \uparrow \infty$ as $i \to \infty$, a.s. for all $m, n$ with $m < n$ we have $X^m_t = X^n_t, t < \theta_m^n$, where
\[\theta_m^n := \inf\{t\geq0:\ |X^m_t|\ \text{Or}\ |X^n_t|\geq m\}.\]
In particular, 
$\theta_m^n=\inf\{t\geq0:\ |X^m_t|\geq m\}=\inf\{t\geq0:\ |X^n_t|\geq m\}$. As such, $\theta_m^n$ is independent of $n(>m)$. Define $\theta_m:=\inf\{t\geq0:\ |X^m_t|\geq m\}$ and set
\[X_t:=
\begin{cases}
X_t^m\ \ \text{for}\ t\leq\theta_m,\\
\infty,\ \text{for}\ t \geq \eta,
\end{cases}\]
so that $(\{X_t\},\eta)$ is a solution of \eqref{fd-sde-sln} for $t < \eta:= \lim_{m\uparrow\infty}\theta_m$.

To prove the uniqueness, we consider the solution $(\{X_t\}, \eta)$ constructed above and compare it with any arbitrary solution $(\{X'_t\}, 
\eta')$ of \eqref{fd-sde-sln}. In the proof of existence of solutions, we had compared $\{X^m_t\}$ and $\{X^n_t\}$. We follow the same approach and define 
\[\theta^R:=\inf\{t\geq0:|X_t|\ \text{Or}\ |X'_t|\geq R\}\wedge\eta\wedge\eta', \forall R \in \N.\]
We then conclude a.s. $X_t = X'_t, t < \theta^R, \forall R\in \N$. Letting $R$ go to infinity concludes the proof.
\end{proof}
                                 
\begin{remark}\label{explicit-conditions}
The `local Lipschitz' condition \eqref{Lipschitz-condition-rnm-inl-nrmsqr-fnl} follows from regularity assumptions on $\sigma, b$ and $F$, provided other hypotheses are satisfied (see \cite[Proposition 3.7]{Levy-SPDE}). As mentioned in Section \ref{S:1}, the class of SDEs \eqref{fd-sde-sln} considered above are related to a class of stochastic PDEs taking values in $\Sc^\prime$. The existence and uniqueness problems for these stochastic PDEs are studied in \cite{Levy-SPDE}.
\end{remark}

\textbf{Acknowledgement:} The first author would like to acknowledge the fact that he was supported by the NBHM (National Board for Higher Mathematics, under Department of Atomic Energy, Government of India) Post Doctoral Fellowship. The second author would like to acknowledge the fact that he was partially supported by the ISF-UGC research grant. The authors would like to thank Prof. B. Rajeev, Indian Statistical Institute Bangalore Centre, India for valuable suggestions during the work.

\bibliographystyle{plain}

\end{document}